\documentclass[10pt, leqno]{article}

\usepackage{mathrsfs}
\usepackage{amsmath}
\usepackage{amssymb}
\usepackage{amsthm}
\usepackage{comment}
\usepackage[dvipdfmx]{hyperref}

\usepackage{accents}

\newtheorem{prop}{Proposition}[section]

\numberwithin{equation}{section}

\def\ep{\epsilon}
\def\dif{\partial}

\begin{document}

\title{
A note on  time functions associated with effectively hyperbolic double characteristics}

\author{Tatsuo Nishitani\footnote{Department of Mathematics, Osaka University:  
nishitani@math.sci.osaka-u.ac.jp
}}

\date{}
\maketitle

\def\dif{\partial}
\def\al{\alpha}
\def\be{\beta}
\def\ga{\gamma}
\def\om{\omega}
\def\lam{\lambda}
\def\tika{{\tilde \nu}}
\def\baka{{\bar \nu}}
\def\varep{\varepsilon}
\def\tal{{\tilde\alpha}}
\def\tbe{{\tilde\beta}}
\def\tis{{\tilde s}}
\def\bas{{\bar s}}
\def\R{{\mathbb R}}
\def\N{{\mathbb N}}
\def\C{{\mathbb C}}
\def\Q{{\mathbb Q}}
\def\Ga{\Gamma}
\def\La{\Lambda}
\def\lr#1{\langle{#1}\rangle_{\gamma}}
\def\lrD{\langle{D}\rangle}
\def\mD{\lr{ D}_{\mu}}
\def\xim{\lr{\xi}_{\mu}}
\def\co{{\mathcal C}}
\def\op#1{{\rm op}({#1})}

\begin{abstract}
Geometrical aspects of effectively hyperbolic singular points over $t=0$ are discussed assuming that the characteristic roots are real only on the one side $t\geq 0$.  In particular, the difference from the case that the characteristic roots are real in both sides $t\geq 0$ and $t <0$ is concerned.
\end{abstract}


\section{Introduction}

Consider
\begin{equation}
\label{eq:moto:op}
P=-D_t^2+A_2(t, x, D)+A_0(t, x, D)D_t+A_1(t, x, D)
\end{equation}
where $A_j(t, x, D)$ are classical pseudodifferential operators of order $j$. Denote  the  principal symbol by
\[
p(t, x, \tau, \xi )=-\tau^2+a(t, x, \xi )
\]
where $a(t, x, \xi)$ is positively homogeneous of degree $2$ in $\xi$, smooth  in $(-T, T)\times U\times (\R^d\setminus 0)$ and satisfies
\begin{equation}
\label{eq:half:hyp}
a(t, x, \xi)\geq 0,\quad (t, x, \xi)\in [0,T_1)\times U\times \R^d
\end{equation}
with some $T_1>0$ and some neighborhood $U$ of $0\in \R^{d}$. Note that if $(0, 0, \tau, \xi)$, $(\tau, \xi)\neq 0$ is a singular point of $p=0$ then $\tau=0$ and $a(0, 0, \xi)=0$.
In what follows we always assume that a singular point $(0, 0, 0, {\bar\xi})$, ${\bar\xi}\neq 0$ is effectively hyperbolic, that is the Hamilton map (e.g. \cite{IP}, \cite{Hobook})
\[
F_p=\frac{1}{2}\begin{pmatrix}\partial^2p/\partial x\partial\xi& \partial^2p/\partial \xi\partial\xi\\
-\partial^2p/\partial x\partial x&-\partial^2p/\partial \xi\partial x
\end{pmatrix}
\]
has nonzero real eigenvalues at $(0,0,0,{\bar \xi})$. 
Assuming a slightly stronger assumption than \eqref{eq:half:hyp} such that
\[
a(t, x, \xi)\geq 0,\quad (t, x, \xi)\in (-\delta_1,T_1)\times U\times \R^d
\]
with some $\delta_1>0$ one can find a smooth function $\varphi(x,\xi)$  in some conic neighborhood $V$ of $(0,{\bar\xi})$, homogeneous of degree $0$ in $\xi$,  and constants $0<\kappa<1$, $c>0$, $\delta>0$  satisfying
\[
a(t, x, \xi)\geq c(t-\varphi(x,\xi))^2|\xi|^2,\quad
\{\varphi, a\}^2\leq 4\kappa a
\]
in $|t|<\delta$, $(x,\xi)\in V$ (\cite[Lemma 1.2.2]{Ni:book}) where $\{\varphi, a\}$ denotes the Poisson bracket of $\varphi$ and $a$. This is the key to proving that the Cauchy problem for $P$ with Cauchy data on $t=0$ is $C^{\infty}$ well-posed for any lower order term in \cite{Ni2, Ni:book}.

To derive energy estimates under the present assumption \eqref{eq:half:hyp} we need some modifications in the arguments in \cite{Ni2, Ni:book}. Our aim in this note is to prove

\begin{prop}
\label{pro:niji:jiko}Assume \eqref{eq:half:hyp}. If a singular point  $(0, 0, 0, {\bar\xi})$ of $p=0$  is effectively hyperbolic then there exist a smooth function $\varphi(x,\xi)$ in a conic neighborhood $V$ of $(0,{\bar\xi})$, homogeneous of degree $0$ in $\xi$, and constants $0<\kappa<1$, $c>0$, $\delta>0$  such that
\begin{equation}
\label{eq:seigen:niji}
a(t, x, \xi)\geq c\min{\big\{t^2, (t-\varphi(x,\xi))^2\big\}}|\xi|^2,\quad
\{\varphi, a\}^2\leq 4\kappa a
\end{equation}
for $ (t, x,\xi)\in [0,\delta)\times V$ 
where  $\varphi(x,\xi)$ satisfies $\big|\dif_x^{\al}\dif_{\xi}^{\be}\varphi\big|\precsim \langle{\xi}\rangle^{-|\be|}$ for any $\al,\be\in\N^d$.
\end{prop}
Note that the condition $\{\varphi, a\}^2\leq 4\kappa a$ with $\kappa<1$ in \eqref{eq:seigen:niji} implies  
that $f=t-\varphi(x, \xi)$ is a time function at $\rho=(0, 0, 0, {\bar \xi})$ for $p$ in the following sense \eqref{eq:time:fun}: Write $\rho'=(0, 0, {\bar\xi})$ and denote the quadratic part of $a((t, x, \xi)+\rho')$ by $Q(t, x, \xi)$ such that $
p((t, x, \tau, \xi)+\rho)=-\tau^2+Q(t, x, \xi)+O(|(t, x, \xi)|^3)$. Denoting $p_{\rho}=-\tau^2+Q(t, x, \xi)$ we have
\begin{equation}
\label{eq:time:fun}
p_{\rho}(-H_f(\rho'))<0
\end{equation}
where $H_f=(0, -\dif \varphi/\dif \xi, -1, \dif \varphi/\dif x)$ is the Hamilton vector field of $f$. To check \eqref{eq:time:fun}, noting that the condition $\{\varphi, a\}^2\leq 4\kappa a$ is invariant under symplectic change of coordinates $(x, \xi)$, one can assume either $\varphi=x_1$ or $d\varphi(0, {\bar \xi})=0$. In the latter case \eqref{eq:time:fun} is obvious. When $\varphi=x_1$ the assertion is also clear if $Q(0, x, \xi)$ contains no $\xi_1$. If not one can write $Q(0, x, \xi)=c(\xi_1-h(x, \xi'))^2+g(x, \xi')$ where $\xi'=(\xi_2,\ldots, x_n)$ with $c\neq 0$. Since we have $0<c\leq \kappa$ for $\{\varphi, a\}^2\leq 4\kappa a$ we conclude \eqref{eq:time:fun}. It is obvious that $f=t$ is also a time function at $\rho$ for $p$. 

Once Proposition \ref{pro:niji:jiko} is proved, applying the same pseudodifferential weight as in \cite{Ni:Ivconj}, one can prove that the Cauchy problem for $P$ with Cauchy data on $t=0$ is $C^{\infty}$ well-posed for any lower order term.

\section{Proof of Proposition \ref{pro:niji:jiko}}

In what follows we denote $x^{(p)}=(x_p, \ldots, x_d)$, $\xi^{(p)}=(\xi_p, \ldots, \xi_d)$, $1\leq p\leq d$ and $t=x_0$.
Note that $(0, 0, 0, {\bar \xi})$ is a singular point of $p=0$ implies 
\begin{equation}
\label{eq:katei}
\dif_t^k\dif_x^{\al}\dif_{\xi}^{\be}a(0, 0, {\bar\xi})=0,\quad k+|\al+\be|=1.
\end{equation}
\begin{prop}
\label{pro:ojm}Assuming the same assumption as in Proposition \ref{pro:niji:jiko}  one can find a homogeneous symplectic coordinates $(x, \xi)$ around $(0, {\bar \xi})$ such that $(0, e_d)=(0, {\bar\xi})$ and $a(t, x, \xi)$ takes the following form \eqref{eq:form:1} with \eqref{eq:form:a} or \eqref{eq:form:2} with \eqref{eq:form:b} and \eqref{eq:form:bbis};
\begin{equation}
\label{eq:form:1}
\begin{split}
&\sum_{i=1}^p(x_{i-1}-x_i)^2q_i(t, x,\xi)+\sum_{i=1}^p\xi_i^2r_i(t, x,\xi)\\
+\big\{(x_p-&\phi_p(x^{(p+1)}, \xi^{(p+1)}))^2
+\psi_p(x^{(p+1)}, \xi^{(p+1)})\big\}q_{p+1}(t, x,\xi),
\end{split}
\end{equation}
\begin{equation}
\label{eq:form:a}
\{\phi_p, \{\phi_p, \psi_p\}\}(0, 0,e_d)=0,
\end{equation}
where $0\leq p\leq d-1$ and
\begin{equation}
\label{eq:form:2}
\sum_{i=1}^p(x_{i-1}-x_i)^2q_i(t, x,\xi)+\sum_{i=1}^p\xi_i^2r_i(t, x,\xi)
+g_p(x^{(p)}, \xi^{(p+1)})r_{p}(t, x,\xi),
\end{equation}
\begin{equation}
\label{eq:form:b}
\{\xi_p, \{\xi_p, g_p\}\}(0, 0,e_d)=0,
\end{equation}
\begin{equation}
\label{eq:form:bbis}
\sum_{i=1}^pr_i^{-1}(0, 0, e_d)>1
\end{equation}
where $1\leq p\leq d-1$.  
In both cases $q_i(t, x,\xi)$, $r_i(t, x,\xi)$ are homogeneous of degree $2$, $0$ respectively in $\xi$, positive at $(0, 0, e_d)$ and $\phi_p$, $\psi_p$, $g_p$ are  homogeneous of degree $0$, $0$, $2$ respectively in $\xi$ vanishing at $(0, 0, e_d)$.
\end{prop}
\begin{proof}
It is enough to repeat exactly the same arguments as the proof of \cite[Theorem 1.1]{Ni2:bis} with minor changes.  After a linear change of coordinates $x$ one can assume that $(0, {\bar\xi})=(0, e_d)$.  If $\dif_t^2a(0, 0, e_d)=0$ then the Taylor expansion of $a(t, x, \xi+e_d)$ at $(t, x, \xi)=(0, 0, 0)$ gives
\[
a(\ep^2t, \ep^3x,\ep^3\xi+e_d)=\ep^5 t\sum_{|\al+\be|=1}\dif_t\dif_x^{\al}\dif_{\xi}^{\be}a(0, 0, e_d)x^{\al}\xi^{\be}+O(\ep^6)\;\;(\ep\to 0)
\]
which proves that $\dif_t\dif_x^{\al}\dif_{\xi}^{\be}p(0, 0, 0, e_d)=0$ for $|\al+\be|=1$ since the left-hand side is nonnegative for small $t\geq 0$ and $(x, \xi)$ near $(0, 0)$. Then it is easy to see that $F_p(0, 0, 0, e_d)$ has only pure imaginary eigenvalues, contradicting to that $(0, 0, 0, e_d)$ is effectively hyperbolic. Thus we conclude  $\dif_t^2a(0, 0, e_d)\neq 0$. 

We first prove that we have either \eqref{eq:form:1} with \eqref{eq:form:a} or \eqref{eq:form:2} with \eqref{eq:form:b}. After that we show \eqref{eq:form:bbis} if the case \eqref{eq:form:2} with \eqref{eq:form:b} occurs. From the Malgrange preparation theorem, one can write
\[
a(t, x, \xi)=\big\{(t-\phi_0(x^{(1)}, \xi^{(1)}))^2+\psi_0(x^{(1)}, \xi^{(1)})\big\}q_1(t, x^{(1)}, \xi^{(1)})
\]
where $\phi_0$, $\psi_0$ are homogeneous of degree $0$ vanishing at $(0, e_d)$ and $q_1$ is homogeneous of degree $2$, $q_1(0,0,e_d)\neq 0$. Since $a(t, 0, e_d)=t^2q_1(t,0,e_d)$ it follows that $q_1(0,0, e_d)>0$. Hence if $\{\phi_0,\{\phi_0, \psi_0\}\}(0, e_d)=0$ this is just \eqref{eq:form:1} with \eqref{eq:form:a} with $p=0$. We go on to the induction on $p$. Assume that \eqref{eq:form:1} holds with $p-1$ while \eqref{eq:form:a} with $p-1$ fails. Set $X_p(x^{(p)}, \xi^{(p)})=\phi_{p-1}(x^{(p)}, \xi^{(p)})$. Note that $d\phi_{p-1}$ and $\sum_{j=p}^d\xi_jdx_j$ are linearly independent at $(0,e_d)$. In fact if not we would have $
\{\phi_{p-1},\{\phi_{p-1},\psi_{p-1}\}\}(0,e_d)=0
$ 
thanks to Euler's identity, which contradicts the assumption. Thus  we can find a homogeneous symplectic coordinates $\{X_j(x^{(p)}, \xi^{(p)}),\, \Xi_j(x^{(p)}, \xi^{(p)})\}_{j=p}^d$ (e.g. \cite[Theorem 21.1.9]{Hobook}) such that
\begin{equation}
\label{eq:sym:cor}
X _j(0,e_d)=0,\;p\leq j\leq d,\;\;\Xi_j(0,e_d)=0,\;p\leq j\leq d-1,\;\Xi_d(0,e_d)\neq 0.
\end{equation}
Denoting $\{X_j, \,\Xi_j\}_{j=p}^d$ by $\{x_j,\xi_j\}_{j=p}^d$ again and noting that
$\dif_{\xi_p}^2\psi_{p-1}(0,e_d)\neq 0$ thanks to the Malgrange preparation theorem one can write
\[
\psi_{p-1}(x^{(p)}, \xi^{(p)})=\big\{(\xi_p-h_p(x^{(p)}, \xi^{(p+1)}))^2+g_p(x^{(p)},\xi^{(p+1)})\big\}b_p(x^{(p)}, \xi^{(p)})
\]
where $b_p$ is of homogeneous of degree $-2$ with $b_p(0,e_d)\neq 0$ and $h_p$ and $g_p$ are homogeneous of degree $1$, $2$ respectively, vanishing at $(0, e_d)$. Take $x=0$ and $\xi_j=0$ unless $j=p, d$ and $\xi_d=1$ then we have
\[
a(0, 0, \xi)=\xi_p^2b_p(0,e_d)q_p(0, 0, \xi)\geq 0
\]
which implies that $b_p(0, e_d)>0$ for $q_p(0, 0, e_d)>0$. Set
\[
\Xi_p(x^{(p)}, \xi^{(p)})=\xi_p-h_p(x^{(p)}, \xi^{(p+1)}),\quad X_p(x^{(p)}, \xi^{(p)})=x_p.
\]
It is clear that $\{\Xi_p, X_p\}=1$ and the differentials $\sum_{j=p}^d\xi_jdx_j$, $d\Xi_p$ and $dX_p$ are linearly independent at $(0,e_d)$. Indeed if $dx_d=\al d\Xi_p+\be dX_p$ with some $\al, \be$ then applying $H_{X_p}$ to this relation we conclude $\al=0$ hence $dx_d=\be dX_p$ which is a contradiction because $p\leq d-1$. Again from \cite[Theorem 21.1.9]{Hobook} one can extend $\Xi_p, \,X_p$ to a homogeneous symplectic coordinates $\{X_j(x^{(p)}, \,\xi^{(p)}), \Xi_j(x^{(p)}, \xi^{(p)})\}_{j=p}^d$ verifying \eqref{eq:sym:cor}. Since $0=\{\xi_j, x_p\}=\{\xi_j, X_p\}=-\dif\xi_j/\dif \Xi_p$, $p+1\leq j\leq d$ and $0=\{x_j, x_p\}=\{x_j, X_p\}=-\dif x_j/\dif \Xi_p$ we see that $\xi_j(X^{(p)}, \Xi^{(p)}),\, \;p+1\leq j\leq d$ and $x_j(X^{(p)}, \Xi^{(p)}),\;p\leq j\leq d$ are independent of $\Xi_p$. Thus we obtain \eqref{eq:form:2} with $p$ where $r_p=b_pq_p$ which is positive at $(0,e_d)$. Now assume that \eqref{eq:form:b} with $p$ fails. Then one can write
\[
g_p(x^{(p)}, \xi^{(p+1)})=\big\{(x_p-\phi_p(x^{(p+1)}, \xi^{(p+1)})^2+\psi_p(x^{(p+1)}, \xi^{(p+1)})\big\}c_p(x^{(p)}, \xi^{(p+1)})
\]
where $c_p$ is homogeneous of degree $2$ with $c_p(0,e_d)\neq 0$ and $\phi_p$, $\psi_p$ are homogeneous of degree $0$, vanishing at $(0, e_d)$. Choose $x^{(p+1)}=0$, $\xi=e_d$ and $x_p=\cdots=x_1=x_0(=t)\geq 0$ then
\[
a(t, x, \xi)=x_p^2c_p(x_p, 0, e_d)r_p(t, x, \xi)\geq 0
\]
hence $c_p$ is positive at $(0,e_d)$ and so is $q_{p+1}=c_p r_p$ because $r_p(0, 0, e_d)>0$. Thus we conclude that \eqref{eq:form:1} with $p$ holds. Therefore the induction on $p$ proves the assertion.

We now show that if the case \eqref{eq:form:2} with \eqref{eq:form:b} occurs then we have \eqref{eq:form:bbis}. 
Choosing $(t=)x_0=\cdots=x_p\geq 0$ and $\xi_1=\cdots=\xi_p=0$ in \eqref{eq:form:2} we have $
a(t, x, \xi)=g_p(x^{(p)},\xi^{(p+1)})r_p(t, x, \xi)\geq 0$ hence
\[
g_p(x_p, x^{(p+1)},\xi^{(p+1)})\geq 0, \quad x_p\geq 0
\]
because $r_p$ is positive at $(0, 0, e_d)$. Since \eqref{eq:form:b} implies that $\dif_{x_p}^2g_p(0,e_d)=0$ repeating the same argument as in the beginning of the proof of Proposition  \ref{pro:ojm} we see that 
\[
\dif_{x_{\mu}}^{k}\dif_{\xi_{\nu}}^{l}\dif_{x_p}g_p(0,e_d)=0, \;\;p+1\leq \mu, \nu\leq d,\quad k+l=1.
\]
This shows that the corresponding quadratic form at $(0,e_d)$ is
\[
\sum_{i=1}^p{\bar q_i}(x_{i-1}-x_i)^2+\sum_{i=1}^p{\bar r_i}\xi_i^2
+{\bar r_p}\big(\sum_{p+1\leq \mu,\nu\leq d, k+l=2}\dif_{x_{\mu}}^{k}\dif_{\xi_{\nu}}^{l}g_p(0,e_d)x_{\mu}^{k}\xi_{\nu}^{l}\big)
\]
where ${\bar q_i}=q_i(0, 0,e_d)$ and the same for ${\bar r_i}$. Denote the sum of the first two terms by $Q_1$ and the third term by $Q_2$. Since $g_p(0, x^{(p+1)}, \xi^{(p+1)})\geq 0$ then $Q_2$ is positive 
semi-definite and hence the eigenvalues of $F_{Q_2}$ are pure imaginary. It is easy to see that
\[
{\rm det}(\lam+F_p(0,e_d))={\rm det}(\lam+F_{Q_1}){\rm det}(\lam+F_{Q_2}).
\]
On the other hand, a direct computation gives
\[
{\rm det}(\lam+F_{Q_2})=\lam^2\psi(\lam),\quad \psi(0)=-\big(\prod_{j=1}^p4{\bar q_j}\big)\big(\prod_{j=1}^p{\bar r_j}\big)\big(\sum_{j=1}^p{\bar r_j}^{-1}-1\big).
\]
 From \cite{IP} the equation $\psi(\lam)=0$ has only pure imaginary roots except possibly a pair of nonzero real simple roots $\pm\mu$. Therefore $F_p(0, 0, 0, e_d)$ has nonzero real eigenvalues if and only if $\psi(0)<0$, that is \eqref{eq:form:bbis}.
\end{proof}

\noindent
Proof of Proposition \ref{pro:niji:jiko}: Without restrictions we may assume that $(0, {\bar \xi})=(0, e_d)$. Taking the homogeneity in $\xi$ it suffices to show \eqref{eq:seigen:niji} in a small neighborhood of $(0, e_d)$. First consider the case \eqref{eq:form:1}.  Denote $x_a=(x_1,\ldots,x_p)$, $x_b=(x_{p+1},\ldots, x_d)$ and $\xi_a=(\xi_1,\ldots, \xi_p)$, $\xi_b=(\xi_{p+1},\ldots, \xi_d)$ and, by an obvious abuse of notation, we often write $q(t, x, \xi)=q(t, x_a, \xi_a, x_b, \xi_b)$.  Let $\chi(s)\in C^{\infty}(\R)$ be such that $\chi(s)=s$ for $|s|\leq 1$, $|\chi(s)|=2$ for $|s|\geq 2$ and $\chi'(s)\geq 0$. By replacing $x_k$, $\xi_k$ by $\delta\chi(x_k/\delta)$, $\delta\chi(\xi_k/\delta)$, $k=1,\ldots, p$ with a small $\delta>0$ in $q_j(t, x, \xi)$, $r_j(t, x, \xi)$ we may assume that  such obtained  ones, which we denote by the same $q_j$, $r_j$, are defined for all $(x_a, \xi_a)\in \R^{2p}$. 
Considering $a/q_{p+1}$ we may assume that $q_{p+1}=1$. Writing  $(x_b, \xi_b)=z+(0,e_d)$ and $w=(y_a, \eta_a)$   consider
\begin{gather*}
Q(w,t, z, \varep)=\sum_{j=1}^p(y_{j-1}-y_j)^2{q_j}(t, \varep y_a+t,\varep\eta_a, z+(0,e_d))\\
+(y_p+1)^2+\sum_{j=1}^p\eta_j^2\,{r_j}(t, \varep y_a+t,\varep\eta_a, z+(0,e_d))
\end{gather*}
where $y_0=0$ and $\varep y_a+t=\varep y_a+(t,\ldots, t)$ ant the same for $\varep\eta_a+t$.  Note that if we choose $\varep=t-\phi(z)$ with $\phi(z)=\phi_p(z+(0, e_d))$ we have
\begin{gather*}
a(t, \varep y_a+t, \varep \eta_a, z+e_d)
=\varep^2Q(w, t, z, \varep)+\psi(z),\quad \varep=t-\phi(z)
\end{gather*}
where $\psi(z)=\psi_p(z+(0, e_d))$. Writing $\theta=(t, z, \varep)$ and note that
\[
Q(w, 0)=\sum_{j=0}^p(y_{j-1}-y_j)^2{\bar q}_j+(y_p+1)^2+\sum_{j=1}^p\eta^2_j{\bar r}_j
\]
hence $Q(w, 0)$ takes the positive minimum in $\R^{2p}$. Taking into account that $|\delta\chi(s/\delta)|\leq 2\delta$ for all $s\in \R$ we see that  for any $\ep>0$ one can choose  $\delta>0$ such that $\big|Q(w, \theta)-Q(w, 0)\big|\leq \ep\, Q(w, 0)$ for all $w\in \R^{2p}$ if $|t|+|z|<\delta$. Therefore for small $|\theta|$ there is a positive minimum of $Q(w, \theta)$ when $w$ varies in $\R^{2p}$ which is attained in $|w|<B$ with some $B>0$ independent of small $\theta$. Denote
\[
\min_{w\in\R^{2p}}Q(w, \theta)=m(\theta).
\]
Note that the Hessian $\nabla^2_wQ(w, \theta)$ of $Q(w, \theta)$ with respect to $w$ can be written $\nabla^2_wQ(w, \theta)=H+R(w, \theta)$ with a nonsingular constant matrix $H$. Here for any $\ep_1>0$ there exist $\delta>0$, $\delta_1>0$ such that $\|R(w, \theta)\|\leq \ep_1$ if $|\theta|< \delta_1$ and $|w|<B$ since $|(d/ds)^j\delta\chi(\varep s/\delta)|\leq C\varep^j\delta^{j-1}$.  Therefore thanks to the implicit function Theorem there exists a smooth ${\bar w}(\theta)$ near $\theta=(0, 0, 0)$ such that
\[
m(\theta)=Q( {\bar w}(\theta), \theta).
\]
Taking $\varep=t-\phi(z)$ we have
\begin{equation}
\label{eq:kihon}
\begin{split}
a(t, \varep y_a+t, \varep \eta_a, z+(0,e_d))=(t-\phi(z))^2Q(w, t, z, t-\phi(z))+\psi(z)\\
\geq m(t, z, t-\phi(z))(t-\phi(z))^2+\psi(z)=m_1(t, z)(t-\phi(z))^2+\psi(z)
\end{split}
\end{equation}
where we have set $m_1(t, z)=m(t, z, t-\phi(z))$. 

Assume $\phi(z)<0$ and hence $\varep=t-\phi(z)>0$ for $t\geq 0$. Then choosing $w=(y_a, \eta_a)$ so that $(\varep y_a+t, \varep\eta_a)=(x_a, \xi_a)$  one concludes that
\begin{gather*}
a(t, x, \xi)=a(t, \varep y_a+t, \varep\eta_a, z+(0,e_d))\geq m_1(t, z)(t-\phi(z))^2+\psi(z).
\end{gather*}
Moreover choosing $w={\bar w}(t, z, t-\phi(z))$ in \eqref{eq:kihon} we see
\[
m_1(t, z)(t-\phi(z))^2+\psi(z)\geq 0, \quad t\geq 0.
\]
In particular, taking $t=0$ we have
\begin{equation}
\label{eq:t_phi:sita}
m_1(0, z)\phi^2(z)+\psi(z)\geq 0.
\end{equation}
Noting that $m_1(t, z)\geq c_1>0$ we also have
\begin{gather*}
a(t, x, \xi)\geq m_1(t, z)\big(t-\phi(z)\big)^2+\psi(z)\\
=m_1(t,  z)\big(t^2+2t|\phi(z)|+\phi^2(z)\big)+\psi(z)\\
\geq c_1t^2+2c_1t|\phi(z)|+m_1(0, z)\phi^2(z)+\psi(z)\\
+\big(m_1(t,  z)-m_1(0, z)\big)\phi^2(z).
\end{gather*}
Since $|m_1(t,  z)-m_1(0, z)|\leq Ct$, in view of \eqref{eq:t_phi:sita} we see
\begin{equation}
\label{eq:1no1}
a(t, x, \xi)\geq c_1t^2+t|\phi(z)|\big(2c_1-C|\phi(z)|\big)\geq c_1t^2,\quad \phi(z)<0
\end{equation}
in a neighborhood of $(0, 0)$ because $\phi(0)=0$. Next if $\phi(z)\geq 0$ then choosing $x_j=\phi(z)\geq 0$, $0\leq j\leq p$ and $\xi_j=0$, $1\leq j\leq p$ in \eqref{eq:form:1} it follows that $\psi(z)\geq 0$ hence we have
\begin{equation}
\label{eq:1no2}
a(t, x, \xi)\geq \big(t-\phi(z)\big)^2,\quad \phi(z)\geq 0.
\end{equation}
 Thus from \eqref{eq:1no1} and \eqref{eq:1no2} we conclude 
\[
a(t, x, \xi)\geq c\,\min{\{t^2, (t-\phi(z))^2\}}|\xi|^2.
\]
Next estimate the Poisson bracket $\{\phi, a\}$. Recall that  $\{\phi,  \{\phi, a\}\}(0,0, e_d)=\{\phi_p, \{\phi_p, \psi_p\}\}(0,0, e_d)=0$ by \eqref{eq:form:a} since $\phi(z)=\phi_p(z+(0, e_d))$. Then for any $\ep_1>0$ one can find a neighborhood $U$ of $(0, e_d)$ such that
\[
\big|H_{\phi}^2\,a\big|=\big|\{\phi,\{\phi, a\}\}\big|\leq \ep_1,\quad ( x, \xi)\in U
\]
uniformly in $t\geq 0$. Since $a\geq 0$ for $t\geq 0$  thanks to the Glaeser's inequality we see that
\[
\big|H_{\phi}a\big|^2=\big|\{\phi, a\}\big|^2\leq 2\,\ep_1\,a,\quad t\geq 0
\]
which finishes the proof for the case \eqref{eq:form:1}.

Turn to the case \eqref{eq:form:2}. Denote $x_a=(x_1,\ldots, x_{p-1})$, $x_b=(x_{p+1},\ldots, x_d)$ and  $\xi_a=(\xi_1,\ldots, \xi_p)$, $\xi_b=(\xi_{p+1},\ldots, \xi_d)$. As above we extend $q_j$, $r_j$ so that such extended ones, containing $\delta>0$,  are defined for all $(x_a, \xi_a)\in \R^{p-1}\times \R^p$. Considering $a/{r_p}$ we may assume ${r_p}=1$ as before.
Consider 
\begin{gather*}
Q(w, t, z, x_p, \varep)=(y_1+1)^2{ q_1}(t, x_p-\varep y_a, x_p, \varep\eta_a, z+(0,e_d))\\
+\sum_{j=1}^{p-1}(y_j-y_{j+1})^2{ q_j}(t, x_p-\varep y_a, x_p, \varep\eta_a, z+(0,e_d))\\
+\sum_{j=1}^p\eta_j^2\,{ r_j}(t, x_p-\varep y_a, x_p, \varep\eta_a, z+(0,e_d))
\end{gather*}
where $y_p=0$,  $w=(y_a,\eta_a)\in\R^{2p-1}$, $(x_b, \xi_b)=z+(0,e_d)$ and $x_p-\varep y_a=(x_p,\ldots, x_p)-\varep y_a$ as before. Note that if we choose $\varep=t-x_p$ then
\begin{align*}
a(t, x_p-\varep y, x_{p}, \varep \eta, z+(0,e_d))
=\varep^2Q(w, t, z,  x_p, \varep)+g(x_p, z),\;\;\varep=t-x_p
\end{align*}
where $g(x_p,z)=g_p(x_p, x_b, \xi_b)$.  Denoting $\theta=(t, z, x_p, \varep)$ and noting
\[
Q(w,0)=(y_1+1)^2{\bar q}_1+\sum_{j=1}^{p-1}(y_j-y_{j+1})^2{\bar q}_j+\sum_{j=1}^p\eta_j{\bar r}_j
\]
one can repeat a similar argument as above to conclude that
\[
\min_{w\in\R^{2p-1}}Q(w, \theta)=m(\theta)=Q({\bar w}(\theta), \theta)
\]
where  ${\bar w}(\theta) $  is smooth near $\theta=0$. 
Choosing $\varep=t-x_p$ we have
\begin{equation}
\label{eq:kihon:2}
\begin{split}
a(t, x_p-\varep y_a, x_p, \varep \eta_a, z+(0,e_d))\\=(t-x_p)^2Q(w, t, z, x_p, t-x_p)+g(x_p, z)\\
\geq m(t, z, x_p, t-x_p)(t-x_p)^2+g(x_p, z)\\
=m_1(t, x_p, z)(t-x_p)^2+g(x_p, z)
\end{split}
\end{equation}
where we have set $m_1(t, x_p, z)=m(t, z, x_p, t-x_p)$. When $x_p<0$, repeating the same arguments as above one can find $c_1>0$
\begin{equation}
\label{eq:2no1}
a(t, x, \xi)\geq c_1t^2+t|x_p|(2c_1-C|x_p|)\geq c_1t^2,\quad x_p<0
\end{equation}
in a neighborhood of $(0, e_d)$. Assume $x_p\geq 0$. Thanks to Proposition \ref{pro:ojm} one has $\sum_{i=1}^p1/{\bar r}_i>1$ then one can find $\ep_i>0$ such that $
\sum_{i=1}^p\ep_i^2\,{\bar r}_i=\rho<1$ with $\sum_{i=1}^p\ep_i=1$. 
Define 
\[
\phi(x)=\sum_{i=1}^p\ep_i x_i.
\]
Since $x_{i-1}-x_i=0$, $i=1, \ldots, p$ implies $t-\phi(x)=x_0-x_0=0$ hence $t-\phi(x)$ is a linear combination of $x_{i-1}-x_i$ so that
\[
t-\phi(x)=\sum_{i=1}^p\alpha_i(x_{i-1}-x_i),\quad \al_i\in\R.
\]
Therefore it is clear that there is $C>0$ such that
\[
(t-\phi(x))^2\leq C\sum_{i=1}^p(x_{i-1}-x_i)^2{ q_i}.
\]
Since $g(x_p, z)\geq 0$ for $x_p\geq 0$ which follows from \eqref{eq:kihon:2} taking $x_p=t\geq 0$ we have $a(t, x,\xi)\geq \sum_{i=1}^p(x_{i-1}-x_i)^2{ q_i}$ which implies that there is $c'>0$ such that
\begin{equation}
\label{eq:2no2}
a(t, x,  \xi)\geq c'\, (t-\phi(x))^2,\quad x_p\geq 0.
\end{equation}
Thus from \eqref{eq:2no1} and \eqref{eq:2no2} we have 
\[
a(t, x, \xi)\geq c\,\min{\{t^2, (t-\phi(x))^2\}}|\xi|^2.
\]
It is clear that for any $\ep>0$ one can find a neighborhood $U$ of $(0, e_d)$ such that
\begin{gather*}
\big|H_{\phi}^2a\big|\leq 2\sum_{i=1}^p\ep_i^2\,{\bar r}_i+\ep=2\rho+\ep,\quad (x, \xi)\in U
\end{gather*}
uniformly in small $t\geq 0$ because $H_{\phi}^2a(0, 0, e_d)=2\sum_{i=1}^p\ep_i^2\,{\bar r}_i$. 
Since $a\geq 0$ for $t\geq 0$ it follows from the Glaeser's inequality that
\[
\big|H_{\phi}a\big|^2=\big|\{\phi, a\}\big|^2\leq 2(2\rho+\ep)a,\quad ( x, \xi)\in U
\]
for small $t\geq 0$ where one can assume $2\rho+\ep<2$ for $\rho<1$. Thus we have completed the proof for the case \eqref{eq:form:2}.
%


\end{document}